\documentclass[12pt,paper]{article}
\usepackage{epsfig,psfrag,amsmath,amssymb,latexsym}
\usepackage{amscd}
\usepackage[dvips]{color}
\usepackage{float}
\usepackage[T1]{fontenc}
\usepackage{subfig}
\usepackage{amssymb, graphicx, multicol, color, epic, eepicemu, epsfig, float, enumerate}
\usepackage{lscape}
\usepackage{amsthm}
\usepackage{natbib}
\usepackage[hypertex]{hyperref}
\newtheorem{theorem}{Theorem}[]

\newtheorem{lemma}{Lemma}[]
\newtheorem{proposition}{Proposition}[]

\newtheorem{definition}[]{Definition}

\newtheorem{remark}{Remark}[]
\newcommand{\psig}{\psi_n(\xi_i,\theta)}

\newcommand{\psion}{\psi_n(\xi_i,\theta_0^{(n)})}

\newcommand{\psigK}{\psi_n(\xi_i,\theta,K)}

\newcommand{\psionK}{\psi_n(\xi_i,\theta_0^{(n)},K)}

\newcommand{\psigKc}{\psi_n(\xi_i,\theta,K^c)}
\newcommand{\psigKcj}{\psi_{n,j}(\xi_i,\theta,K^c)}

\newcommand{\psionKc}{\psi_n(\xi_i,\theta_0^{(n)},K^c)}
\newcommand{\psionKcj}{\psi_{n,j}(\xi_i,\theta_0^{(n)},K^c)}

\newcommand{\thg}{\theta}
\newcommand{\la}{\lambda}

\newcommand{\Thn}{\Theta_n}
\newcommand{\tho}{\theta_0}
\newcommand{\thon}{\theta_0^{(n)}}

\newcommand{\ltgK}{\lambda_n(\theta,K)}

\newcommand{\juur}{\frac{1}{\sqrt{n}}}
\begin{document}
\thispagestyle{empty}
\setlength{\parindent}{0cm}
\textbf{
\Large{\centerline{Asymptotic normality of generalized maximum spacing}
	\centerline{estimators for multivariate observations}}}
%
%
\setlength{\parindent}{18pt}
\newline \newline
\centerline{Kristi Kuljus$^a$ and Bo Ranneby$^b$}
\newline \newline
\centerline{$^a$Institute of Mathematics and Statistics, University of Tartu}
\centerline{$^b$Department of Forest Economics, Swedish University of Agricultural Sciences}
\newline
\begin{abstract}
In this paper, the maximum spacing method is considered for multivariate observations.
Nearest neighbour balls are used as a multidimensional analogue to univariate spacings. A class of information-type measures
is used to generalize the concept of maximum spacing estimators. Asymptotic normality of these generalized maximum spacing estimators is proved when the assigned model class is correct,
that is the true density is a member of the model class.

\end{abstract}
\section{Generalized maximum spacing estimators}
\subsection{Introduction}
For independent and identically distributed univariate observations a new estimation method, the maximum spacing (MSP) method,
was defined in \citet{Ranneby84} and independently by \citet{ChengAmin}. In \citet{RanJamTet}, the MSP method was extended to multivariate observations for the Kullback-Leibler information measure. In \citet{KR2015}, the multivariate maximum spacing estimation method based on nearest neighbour balls was considered for a broader class of information-type measures. Weak and strong consistency of these generalized maximum spacing (GMSP) estimators under general conditions was proved. In the univariate case such GMSP estimators based on different metrics were studied in \citet{RanEkstr97}, \citet{Ekstr2001} and \citet{GhoshJam},
 in the last work also asymptotic normality of GMSP estimates was proved. Consistency and asymptotic normality of GMSP estimates in the univariate case was also considered in \citet{luong}.
As exemplified already in \citet{Ranneby84}, an advantage of the maximum spacing method compared to the maximum likelihood method is the possibility of checking the validity of the assigned model class at the same time with solving the estimation problem. In \citet{KR2015} it was demonstrated that combining information from spacing functions under different divergence measures can provide further insight in the model validation context. In the present paper we study asymptotic normality of GMSP estimators for information-type measures considered in \citet{KR2015}.
\subsection{Notation and definitions}
Let $\xi_1,\ldots,\xi_n$ be a sequence of independent and identically distributed $d$-dimensional random vectors with distribution $P_0$ that is absolutely continuous with respect to Lebesgue measure. Let the corresponding density function be $g(x)$. Define the nearest neighbour distance to the random variable $\xi_i$ as
\[R_n(i)= \min_{j\neq i} |\xi_i-\xi_j| \, , \quad \quad i=1,\ldots,n \, .\]
Let $B(x,r)=\{ y:\,|x-y|\le r\}$ denote the ball of radius $r$ and center $x$. Let $\mbox{NN}_i$ denote the nearest neighbour of $\xi_i$ and let $B_n(\xi_i)$ denote its nearest neighbour ball, i.e.~this is a ball with center $\xi_i$ and radius $R_n(i)$. Suppose we assign a model with density functions $\{f_\theta (x), \, \theta\in \Theta\}$, where
$\Theta \subset \mathbb{R}^q$, and assume that the true density $g(x)$ belongs to the family with the parameter vector given by $\theta_0$. Define random variables $z_{i,n}(\theta)$ as
\[ z_{i,n}(\theta)=nP_\theta(B_n(\xi_i))=n\int_{B_n(\xi_i)} f_{\thg}(y)dy, \quad \quad i=1,\ldots,n \, .\]
In \citet{KR2015} the maximum spacing method was generalized to multivariate observations for strictly concave functions $h: (0,\infty) \rightarrow (-\infty,0]$ with maximum at $x=1$.
The generalized maximum spacing function $S_n(\theta)$ was defined as
\[S_n(\theta) =\frac{1}{n} \sum_{i=1}^n h(z_{i,n} (\theta)) \, .\]
\begin{definition}
The parameter value that maximizes $S_n(\theta)$ is called the generalized maximum spacing estimate (GMSP estimate) of $\theta$
 and denoted by $\hat{\theta}_n$.
If $\sup_{\theta} S_n(\theta)$ is not attained for any $\theta$ in the admissible set $\Theta$, then the GMSP estimate $\hat{\theta}_n$ is defined as any point of $\Theta$ that satisfies
\[ S_n(\hat{\theta}_n)\ge -c_n +\sup_{\theta \in \Theta} S_n(\theta)\, ,\]
where $c_n>0$ is a sequence of constants such that $c_n\to 0$ as $n\to \infty$.
\end{definition}
\noindent Examples of functions $h$ satisfying the conditions given above are:
 \[ h_1(x)=\ln x-x+1 , \quad h_2(x)=(1-x)\ln x , \quad h_3(x)=-|1-x^{1/p}|^p , \]
 \[ h_4(x)=-|1-x|^p , \quad h_5 (x) = \mbox{sgn}(1-\alpha)(x^{\alpha}-\alpha x +\alpha -1),  \]
where $\alpha>0$, $\alpha \neq 1$, and $p\geq 1$.
Here $h_2$ corresponds to Jeffreys' divergence measure, $h_3$ to the Hellinger distance, $h_4$ to Vajda's measure of information and
$h_5$ to R\'{e}nyi's divergence measure. In this article, we will consider only $p=2$ for function families $h_3$ and $h_4$. For $h_5$, we restrict $\alpha$ to
$(0,1)\cup (1,2]$. Observe that $\alpha=1/2$ corresponds to $h_3$ with $p=2$ and $\alpha=2$ corresponds to $h_4$ with $p=2$. Thus, $h_3$ and $h_4$ with $p=2$ will be covered by $h_5$.

To prove asymptotic normality of $\hat{\theta}_n$, we will work with the partial derivatives of $h(z_{i,n}(\thg))$, the vector of partial derivatives is denoted by $\psig$. Let
$v(z_{i,n}(\thg))=h'(z_{i,n}(\thg))z_{i,n}(\thg)$. Then
\[ \psig= v(z_{i,n}(\thg)) \frac{1}{P_{\thg}(B_n(\xi_i))} \int_{B_n(\xi_i)} \nabla f_{\thg} (y) dy.    \]
Define $\tilde{\psi}_n(\xi_i,\thg)$ and ${\psi}(\xi_i,\thg)$ as follows:
\[ \tilde{\psi}_n(\xi_i,\thg)=v\left( z_{i,n}(\theta_0) \frac{f_{\thg}(\xi_i)}{g(\xi_i)} \right) \frac{\nabla f_{\thg}(\xi_i)}{f_{\thg}(\xi_i)}, \quad
 {\psi}(\xi_i,\thg)= v \left( Z \frac{f_{\thg}(\xi_i)}{g(\xi_i)} \right) \frac{\nabla f_{\thg}(\xi_i)}{f_{\thg}(\xi_i)}. \]
Observe that $z_{i,n}(\thg)$ can be written as
\[ z_{i,n}(\thg)= z_{i,n}(\theta_0) \frac{1}{P_{\theta_0}(B_n(\xi_i))} \int_{B_n(\xi_i)} f_{\thg} (y) dy. \]
Since
\[ \frac{1}{P_{\theta_0}(B_n(\xi_i))} \int_{B_n(\xi_i)} f_{\thg} (y) dy \stackrel{a.s}{\to} \frac{f_{\thg}(\xi_i)}{g(\xi_i)},\quad
 \frac{1}{P_{\thg}(B_n(\xi_i))} \int_{B_n(\xi_i)} \nabla f_{\thg} (y) dy \stackrel{a.s}{\to} \frac{\nabla f_{\thg}(\xi_i)}{f_{\thg}(\xi_i)}, \]
it follows that $\tilde{\psi}_n(\xi_i,\thg)$ is obtained from $\psig$ by substituting the integral quantities above with their almost sure limits.
Let $\la_n(\thg)= E[ \psig ]$ and $\la(\thg)=E[{\psi}(\xi_i,\thg)]$.
\subsection{Idea for proving asymptotic normality of $\hat{\thg}_n$}
Let $\thon$ denote the point maximizing the expectation function $E[h(z_{i,n}(\thg))]$, thus it satisfies $\la_n(\thon)=0$. Recall that $\tho$ satisfies $\la(\tho)=0$.
Consistency of $\hat{\thg}_n$ implies that a sequence $\{ \delta_n \}_1^{\infty}$ can be chosen so that $\delta_n \to 0$ slowly enough to ensure
\[ P(|\hat{\thg}_n-\theta_0|\ge \delta_n)\to 0 \quad \mbox{as} \quad n\to \infty. \]
We will show that $\la_n(\thg)$ converges uniformly to $\la(\thg)$ in a neighbourhood of $\tho$, thus $\thon \to \tho$ as $n \to \infty$.
Therefore, we will consider shrinking neighbourhoods $\Theta_n=\{ \theta: \, |\theta-\theta_0|\le \alpha_n \}$, where $\delta_n<\alpha_n$ and $\alpha_n\to 0$ is such, that $\thon \in \Thn$ for every $n$.
The key steps for proving asymptotic normality of $\hat{\thg}_n$ are the following.
\newline \newline
\noindent \textbf{Step 1}. First we will show that $$\frac{1}{\sqrt{n}} \sum_{i=1}^n \tilde{\psi}_n(\xi_i,\tho)$$ is asymptotically normally distributed. To prove asymptotic normality of this quantity, we will interpret it as a function of a weighted empirical process which converges weakly to a Gaussian process.
\newline \newline
\noindent \textbf{Step 2}. We will prove that
\[ \frac{1}{\sqrt{n}} \sum_{i=1}^n \tilde{\psi}_n(\xi_i,\tho) - \frac{1}{\sqrt{n}} \sum_{i=1}^n \psion \stackrel{p}{\to} 0. \]
Thus, $\frac{1}{\sqrt{n}} \sum_i \psion$ has the same asymptotic distribution as $\frac{1}{\sqrt{n}} \sum_i \tilde{\psi}_n(\xi_i,\tho)$.
\newline \newline
\noindent \textbf{Step 3}. To finalize proving asymptotic normality of $\hat{\thg}_n$, we will follow the approach by \citet{huber67} and show that
\begin{equation} \label{HuberConv}
\frac{1}{\sqrt{n}} \sum_{i=1}^n \psion +\sqrt{n}\la_n(\hat{\thg}_n) \stackrel{p}{\to} 0.
\end{equation}
Expanding $\la_n(\hat{\thg}_n)$ around $\thon$ then gives that $\sqrt{n}(\hat{\thg}_n-\thon)$ is asymptotically normally distributed.
\newline \newline
\noindent Crucial for proving (\ref{HuberConv}) is Lemma 3 in \citet{huber67} stating that
\begin{equation} \label{HuberL3}
 \sup_{\Thn} \frac{\left|\juur \sum_i(\psig-\psion)-\sqrt{n}\lambda_n(\thg)\right| }{1+\sqrt{n}|\lambda_n(\thg)| } \stackrel{p}{\to} 0.
\end{equation}
Since our assumptions imply that $\la_n(\thg)$ is continuously differentiable in a neighbourhood of $\tho$ with a negative definite derivative matrix $V_n(\thg)$, there exists $C>0$ such that
$ |\la_n (\thg)| \ge C|\thg -\thon| $ when $n$ is large enough.
Thus, the convergence in (\ref{HuberL3}) follows if
\[  \sup_{\Thn} \frac{\left|\juur \sum_i(\psig-\psion)-\sqrt{n}\lambda_n(\thg)\right| }{1+\sqrt{n}|\thg-\thon| } \stackrel{p}{\to} 0 \]
holds. Therefore, we will work with expressions having $1+\sqrt{n}|\thg-\thon|$ in the denominator.

Convergence of the weighted empirical process in Step 1 is used to prove asymptotic normality of general functions of the process. Both convergences will be proved in Section 2. In Section 3, this result will be applied for proving asymptotic normality of the approximation of our function of interest, that is $\frac{1}{\sqrt{n}} \sum_{i=1}^n \tilde{\psi}_n(\xi_i,\theta_0)$. Step 2 will also be proved in Section 3. Section 4 deals with Step 3: bracketing technique and a stochastic differentiability condition will be used to prove asymptotic normality of $\hat{\thg}_n$.
\subsection{Assumptions}
We will work with first and second order derivatives of different functions with respect to the components $\theta_j$, $j=1,\ldots,q$,  of the parameter vector $\theta$.
The notations $\partial f_{\thg}(\xi)$ and $\partial^2 f_{\thg}(\xi)$ will be used instead of $\frac{\partial f_{\thg}(\xi)}{\partial \thg_j}$ and $\frac{\partial^2 f_{\thg}(\xi)}{\partial \thg_j \partial \thg_l}$, respectively, when the computations are analogous or a certain condition has to hold independently of $j,l=1,\ldots,q$. Let $\Theta_0$ denote a neighbourhood of $\theta_0$.
Asymptotic normality of $\hat{\theta}_n$ will be proved under different combinations (depending on the function $h$) of the following assumptions:
\begin{itemize}
\item[A1] $\{x\in\mathbb{R}^d: \, g(x)>0\}$ is an open set in $\mathbb{R}^d$, $g$ is uniformly bounded and continuous on $\{ g>0\}$
\item[A2] Assume that $\theta_0$ is an interior point of $\Theta$. Suppose $f_{\thg}(x)$ and its first and second order derivatives with respect to the components of $\thg$ are continuous in $\thg \in \Theta_0$ and in $x$.
\item[A3] \[ E\left[ \sup_{\Theta_0} \left| \frac{\partial f_{\thg}(\xi)}{g(\xi)} \right|^2 \right] < \infty \]
\item[A4] The following random variables are uniformly integrable:
 \[  \sup_{\Theta_0}  \frac{1}{P_{\thg}(B_n(\xi))} \int_{B_n(\xi)}|\partial f_{\thg}(y)|dy  \]
\item[A5] The following random variables are uniformly integrable:
 \[  \left(\frac{1}{P_{\thon}(B_n(\xi))} \int_{B_n(\xi)}|\partial f_{\thon}(y)|dy \right)^2  \]
\item[A6] The following random variables are uniformly integrable:
 \[  \sup_{\Theta_0} \left( \frac{1}{P_{\thg}(B_n(\xi))} \int_{B_n(\xi)}|\partial f_{\thg}(y)|dy \right)^2 \]
\item[A7] \[ E\left[ \sup_{\Theta_0} \left| \frac{\partial f_{\thg}(\xi)}{g(\xi)} \right|^4 \right] < \infty , \quad
             E\left[ \sup_{\Theta_0} \left| \frac{f_{\thg}(\xi)}{g(\xi)} \right|^4 \right] < \infty   \]
\item[A8] The following random variables are uniformly integrable:
\[  \sup_{\Theta_0}  \frac{1}{P_{\thg}(B_n(\xi))} \int_{B_n(\xi)} |{\partial^2 f_{\thg}(y)}|dy   \]
\item[A9] For some $\varepsilon$, $0<\varepsilon<1$,
 \[ E\left[ \sup_{\Theta_0} \left| \frac{\partial^2 f_{\thg}(\xi)}{g(\xi)} \right|^{1+\varepsilon} \right] < \infty \]
\item[A10] \[ E\left[ \sup_{\Theta_0} \left| \frac{\partial^2 f_{\thg}(\xi)}{g(\xi)} \right|^{2} \right] < \infty \]
\end{itemize}
Let $q(x)\le \inf_{\Theta_0}f_{\theta}(x)$ be the Radon-Nikodym derivative of a finite measure $Q$ with respect to the Lebesgue measure on $\mathbb{R}^d$. The following two remarks give two examples of conditions when A4 and A5 are satisfied.
\begin{remark}
Suppose
\[E_q \left( \sup_{\Theta_0} \frac{|\partial f_{\thg}(\xi)|}{q(\xi)} \right)^2<\infty, \quad E_g\left( \frac{g(\xi)}{q(\xi)} \right)<\infty,\]
then A4 holds. If we instead assume $E_q \left( \sup_{\Theta_0} \frac{|\partial f_{\thg}(\xi)|}{q(\xi)} \right)^4<\infty$, then A6 holds.
\end{remark}
\begin{remark}
If for some constants $c_1$, $c_2$ and $\forall x\in\mathbb{R}^d$, $\forall \theta \in {\Theta_0}$,
\[ c_1g(x) \le f_{\theta}(x) \le c_2g(x), \]
then A4 and A5 follow from A3.
\end{remark}
\section{Asymptotic normality of weighted empirical processes}
In this section we will modify the results of \citet{schill83} and prove that a weighted empirical process of $z_{1,n}(\theta_0),\ldots,z_{n,n}(\theta_0)$ converges to a Gaussian process. Using a suitable transformation we then obtain asymptotic normality of
\[ \frac{1}{\sqrt{n}} \sum_{i=1}^n q(z_{i,n}(\theta_0)) u(\xi_i),  \]
where $q(\cdot)$ is a function of bounded variation on $[0,A]$ for any $A>0$, and $u(\cdot)$ is a continuous weight function with the properties $E u(\xi_i)=0$, $Eu^2(\xi_i)<\infty$.
To prove asymptotic normality of the sum above, we use results from \citet{bibr}, \citet{schill83} and \citet{ZhouJam}.
Let
\[ W_{i,n}=ng(\xi_i)V(R_i), \quad \quad i=1,\ldots,n,\]
where $V(r)$ represents the volume of a $d$-dimensional sphere of radius $r$. In \citet{bibr}, it is shown that the normalized (centered and scaled) empirical distribution function of $e^{-W_{1,n}},\ldots,e^{-W_{n,n}}$ converges under the true distribution weakly to a Gaussian process with mean zero and covariance function independent of the true underlying density.  In \citet{schill83}, the same result is proved for a weighted empirical process with a bounded continuous weight function. To be able to use Theorem 2.2 from \citet{schill83}, we will study truncated weight functions $u_N(\cdot)$ defined as follows.
Since
$$Eu(\xi)=\int u^+(x) dP_0(x) - \int u^-(x) dP_0(x)=0,$$ we can find for every $N>0$ a constant $N^{\ast}>0$ such that
\[ \int_{\{ u^+(x)\le N \}} u^+(x) dP_0(x) = \int_{\{ u^-(x)\le N^{\ast} \}} u^-(x) dP_0(x). \]
Define a bounded weight function $u_N(x)$ as follows:
\begin{equation} \label{weightfnN} u_N(x) =\left\{ \begin{array}{ll}
                u(x),  & \quad -N^{\ast} \leq u(x) \leq N  , \\
                0, & \quad otherwise.
               \end{array}\right.
\end{equation}
Take $N_1=\max \{N,N^{\ast} \}$, then $Eu_N(\xi)=0$ and $|u_N(\xi)|\le N_1$.

The general ideology for proving the convergence
\[ \frac{1}{\sqrt{n}}\sum_{i=1}^n q(z_{i,n}(\theta_0)) u(\xi_i) \stackrel{D}{\to} \mathcal{N}(0,\sigma_q^2\tau^2), \]
where $\tau^2=E[u^2(\xi_i)]$, will be as follows. We consider bounded weight functions $u_N(x)$ defined as in $(\ref{weightfnN})$ and define the weighted empirical processes $\{Y_n(t): \, 0\le t\le \infty\}$ and
$\{Z_n(t):\, 0\le t\le \infty\}$:
\begin{equation} \label{procZn} Z_n(t)= \frac{1}{\sqrt{n}}\sum_{i=1}^n I(z_{i,n}(\theta_0) > t)u_N(\xi_i), \end{equation}
\[ Y_n(t)=\frac{1}{\sqrt{n}}\sum_{i=1}^n \left\{I(W_{i,n} > t)u_N(\xi_i)-E\left[ I(W_{i,n} > t)u_N(\xi_i) \right] \right\}.  \]
From \citet{schill83} it follows that $\{Y_n(t)\}$ converges weakly to a Gaussian process. For large $n$ we have $ng(\xi_i)V(R_i)=W_{i,n}\approx nP_0(B_n(\xi_i))=z_{i,n}(\theta_0)$. Thus, if we can show that $\{ Z_n(t)\}$ is tight and $\mbox{Var}(Z_n(t)-Y_n(t))\to 0$ for every $t$, then $\{Z_n(t)\}$ converges to the same Gaussian process. Therefore, using the results from
\citet{bibr}, \citet{schill83} and \citet{ZhouJam}, we can show that
\[ Z_n(t) \stackrel{D}{\to} Z(t), \]
where $\{Z(t): \, 0\le t\le \infty\}$ is a Gaussian process with mean zero and with a certain covariance function. We then apply the integral transform
\[ s(x)=q(0)x(0) + \int_0^A x(t) dq(t) \]
to $Z_n(t)$ and obtain via $s(Z_n(t))\stackrel{D}{\to} s(Z(t))$ the desired result.
\begin{proposition} \label{GaussProc}
Suppose A1 and the following conditions hold:
$$|u_N(\xi_i)|\le N_1 \,\, \mbox{for some} \,\, N_1>0, \quad E u_N(\xi_i)=0, \quad E u^2_N(\xi_i)=\tau^2_N. $$
Then $\{Z_n(t): \, 0\le t \le \infty\}$ defined in (\ref{procZn}) converges weakly to a Gaussian process $\{ Z(t): \, 0\le t \le \infty\}$ with mean zero and covariance function
$\tau_N^2 k(s,t)$, where
$k(s,t)$ is given by
\[ k(s,t)= e^{-t} -te^{-s-t} +e^{-s-t}\int_{W(s,t)} (e^{\beta(s,t,x)}-1)dx, \quad 0\leq s \leq t \leq \infty, \]
where
\[ W(s,t)=\{ x\in \mathbb{R}^d: r_1\leq |x|\leq r_1+r_2 \}, \quad
\beta(s,t,x)=\int_{B(0,r_1)\cap B(x,r_2)} dz,\]
with $r_1$ and $r_2$ corresponding to the volumes $t$ and $s$ of the balls $B(0,r_1)$ and $B(0,r_2)$, respectively.
\end{proposition}
\begin{proof} From \citet{schill83} it follows directly that the centered empirical process $\{Y_n(t)\}$ converges weakly to the Gaussian process defined above. To conclude that
$\{Z_n(t)\}$ converges to the same limit, we prove that for every $t$, $\mbox{Var}(Z_n(t)-Y_n(t))\to 0$, and that $\{ Z_n(t)\}$ is tight.

a) That $\mbox{Var}(Z_n(t)-Y_n(t))\to 0$ as $n\to \infty$, follows with minor modifications from \citet{ZhouJam}.
Since
\[\mbox{Var}[Z_n(t)-Y_n(t)]\le E\left[ \left( I(z_{1,n}(\theta_0)>t)-I(W_{1,n}>t) \right)u_N(\xi_1) \right]^2 \]
\[ + n\left|\mbox{Cov}\left[ \left( I(z_{1,n}(\theta_0)>t)-I(W_{1,n}>t) \right)u_N(\xi_1), \left( I(z_{2,n}(\theta_0)>t)-I(W_{2,n}>t) \right)u_N(\xi_2)    \right]\right|, \]
we need to show that
\[ \lim_{n\to \infty} E \left[  I(z_{1,n}(\theta_0)>t)-I(W_{1,n}>t)  \right]^2 =0, \]
\[ \lim_{n\to \infty} n\left|\mbox{Cov}\left[ \left( I(z_{1,n}(\theta_0)>t)-I(W_{1,n}>t) \right)u_N(\xi_1), \left( I(z_{2,n}(\theta_0)>t)-I(W_{2,n}>t) \right)u_N(\xi_2)    \right]\right|=0.\]
The convergence of both terms follows as in the proof of Proposition 1 of \citet{ZhouJam}. For the proof of the convergence of the covariance term, Lemma 2.11 in \citet{bibr} is fundamental.

b) Tightness of $\{Z_n(t)\}$ can be proved similarly to \citet{schill83} and \citet{bibr}. As in \citet{schill83}, we can split $Z_n(t)$ as follows: $Z_n(t)=Z_n^+(t)-Z_n^-(t)$, where
\[ Z_n^{\pm}(t) = \frac{1}{\sqrt{n}} \sum_{i=1}^n \left\{I(z_{i,n}(\theta_0)>t)u_N^{\pm}(\xi_i)- E\left[ I(z_{i,n}(\theta_0)>t)u_N^{\pm}(\xi_i) \right] \right\}. \]
It is enough to show that $\{Z_n^+(t)\}$ is tight.
Let $l_i=I(a\le z_{i,n}(\theta_0)\le b)$, then
\[ El_1 = E\left[ I \left({a}\le n P_0(B_n(\xi_1))\le {b} \right)\right] =\frac{n-1}{n} \int_a^b \left( 1-\frac{w}{n} \right)^{n-2} dw  \]
\[ <e^2 \int_a^b e^{-w} dw =e^2(e^{-a}-e^{-b}),\]
\[ E\left[ l_1 P_0^2(B_n(\xi_1)) \right] =\frac{1}{n^2}\int_a^b  \frac{n-1}{n} w^2 \left( 1-\frac{w}{n} \right)^{n-2} dw \]
\[< \frac{e^2}{n^2} \left( e^{-a}(a^2+2a+2)-e^{-b}(b^2+2b+2) \right). \]
Thus, applying Theorem 2.1 in \citet{bibr} gives that for some constant $M>0$,
\[ E\left[\sum_{i=1}^n (l_i-El_i)\right]^4 < M\left[ n^2(Q(b)-Q(a))^2+n \right],  \]
where $Q(t)$ is a continuous distribution function defined as
\[ Q(t) = 1-\frac{1}{3}e^{-t} (t^2 + 2t +3). \]
The rest of the proof goes according to \citet{schill83} and \citet{bibr}.
\end{proof}
\begin{proposition} \label{normconvN}
Suppose the assumptions of Proposition \ref{GaussProc} hold and $q(t)$ is of bounded variation on $[0,A]$ for each $A>0$. Let $|\int_1^\infty (te^{-t})^{1/2} dq(t)|<\infty$. Then
\[ \frac{1}{\sqrt{n}} \sum_{i=1}^n q(z_{i,n}(\theta_0))u_N(\xi_i) \to \mathcal{N}(0,\tau_N^2 \sigma_q^2),\]
where
\begin{equation} \label{CovConstsigq}
\sigma_q^2= q^2(0)+\int_0^{\infty}\int_0^{\infty} k(s,t) dq(s)dq(t)+2q(0) \int_0^{\infty} k(0,t)dq(t)
\end{equation}
with $k(0,t)=e^{-t}-te^{-t}$.
\end{proposition}
\begin{proof} Recall the definition of the empirical process $Z_n(t)$ in (\ref{procZn}).
In the proof of Proposition 2 in \citet{ZhouJam} it is shown that for some $M>0$,
$\mbox{Var}[Z_n(t)]\le Mt e^{-t}$, $t\ge 1$. Therefore it follows according to our assumption that
$\int_0^{\infty}\left( \mbox{Var}[Z_n(t)] \right)^{1/2}dq(t) < \infty$, which implies (see e.g.~\citet{CramLead67}, p.~90-91) that
for every $n$,
\[ \int_0^A Z_n(t) dq(t) \stackrel{A\to \infty}{\to} \int_0^{\infty} Z_n(t) dq(t). \]
Analogously,
\[ \int_0^A Z(t) dq(t) \stackrel{A\to \infty}{\to} \int_0^{\infty} Z(t) dq(t). \]
Therefore, $\int_0^{\infty} Z_n(t) dq(t)$ is well defined and it holds with probability one that
\[ \frac{1}{\sqrt{n}}\sum_{i=1}^n q(z_{i,n}(\theta_0))u_N(\xi_i)= - \int_0^{\infty} q(t) d Z_n(t)
 =q(0)Z_n(0) + \int_0^\infty Z_n(t) d q(t).\]
Since $q(t)$ is of bounded variation on $[0,A]$ for every $A>0$, it follows from Proposition $\ref{GaussProc}$ that
\[  - \int_0^{A} q(t) d Z_n(t) \stackrel{D}{\to} q(0)Z(0) + \int_0^A Z(t) d q(t). \]
Because
\[ \lim_{A\to \infty} \lim_{n\to \infty} P\left(\left|\int_0^{\infty} Z_n(t) dq(t)-\int_0^A Z_n(t) dq(t)\right|>\varepsilon \right) \]
\[ \le \lim_{A\to \infty} \frac{1}{\varepsilon^2} \mbox{Var}\left( \int_A^{\infty} Z_n(t)dq(t) \right)\le \lim_{A\to \infty} \frac{1}{\varepsilon^2}M \left( \int_A^{\infty}(te^{-t})^{1/2}dq(t) \right)^2 =0, \]
it follows according to Theorem 4.2 in \citet{billingsley68} that
\begin{equation} \label{LimVariable}
q(0)Z_n(0) + \int_0^\infty Z_n(t) d q(t) \stackrel{D}{\to} q(0)Z(0) + \int_0^\infty Z(t) d q(t).  \end{equation}
The variance of the limiting distribution can now be calculated using the covariance function of $Z(t)$. That the random variable on the right hand side of (\ref{LimVariable}) is normally distributed, follows since it is an integral of a normal process.
\end{proof}
\begin{proposition} \label{normconv}
Assume the assumptions of Proposition \ref{normconvN} are valid. Substitute the truncated function $u_N(\xi_i)$ with $u(\xi_i)$ and suppose $\tau^2=Eu^2(\xi_i)<\infty$. Then
\begin{equation} \label{GeneralConv}
 \frac{1}{\sqrt{n}} \sum_{i=1}^n q(z_{i,n}(\theta_0))u(\xi_i) \stackrel{D}{\to} \mathcal{N}(0,\tau^2 \sigma_q^2).\end{equation}
\end{proposition}
\begin{proof}
Let $\Delta_N(\xi_i)=u(\xi_i)-u_N(\xi_i)$ and $z_i=z_{i,n}(\theta_0)$. Then
\[ E\left[ \frac{1}{\sqrt{n}} \sum_{i=1}^n q(z_i)\Delta_N(\xi_i) \right]^2 = E[q^2(z_1)\Delta_N^2(\xi_1)]
 +(n-1)E\left[ q(z_1)\Delta_N(\xi_1) q(z_2)\Delta_N(\xi_2) \right].\]
To prove that the covariance term converges to zero as $N\to \infty$, we use the conditional approach of \citet{Schilling86}. Let $\{ \mbox{NN}_1=\xi_2 \}$ denote the event that the nearest neighbour of $\xi_1$ is $\xi_2$.  Consider the following five mutually exclusive sets for various nearest neighbour geometries of $\xi_1$ and $\xi_2$:
\[ D_1=\{ \mbox{NN}_1=\xi_2, \, \mbox{NN}_2=\xi_1  \},  \,\, D_2=\{ \mbox{NN}_1= \mbox{NN}_2 \}, \,\, D_3=\{ \mbox{NN}_1=\xi_2, \, \mbox{NN}_2 \neq \xi_1  \},  \]
\[ D_4=\{ \mbox{NN}_1 \neq \xi_2, \, \mbox{NN}_2=\xi_1  \}, \, \, D_5=\{ \mbox{NN}_1 \neq \xi_2, \, \mbox{NN}_2 \neq \xi_1, \, \mbox{NN}_1\neq \mbox{NN}_2 \}.\]
Then,
\[ E\left[ q(z_1)\Delta_N(\xi_1) q(z_2)\Delta_N(\xi_2) \right]
 = P(D_5)E[ q(z_1)\Delta_N(\xi_1) q(z_2)\Delta_N(\xi_2) |D_5 ] \]
 \[ + \sum_{i=1}^4 P(D_i) E [ q(z_1)\Delta_N(\xi_1) q(z_2)\Delta_N(\xi_2) |D_i ]. \]
Given $D_5$, we have independence, therefore the covariance is zero. Since for $i=1,\ldots,4$, $P(D_i)=\mathcal{O}(1/n)$, it is sufficient to show that
the conditional expectations tend to zero as $N\to \infty$, $i=1,\ldots,4$. We have
\[ \left| E[q(z_1)\Delta_N(\xi_1) q(z_2)\Delta_N(\xi_2) |D_i] \right| \le E[ q^2(z_1)\Delta_N^2(\xi_1)|D_i ]  \]
\[ =E[q^2(z_1)\Delta_N^2(\xi_1)]= E[q^2(z_1)] E[\Delta_N^2(\xi_1)] \to 0 \quad \mbox{as} \quad N\to \infty. \]
Thus, Theorem 4.2 in \citet{billingsley68} implies (\ref{GeneralConv}).
\end{proof}
\section{Asymptotic normality of the derivative of the GMSP function}
In Proposition \ref{normconv} we proved asymptotic normality for a general function $u(\xi_i)$ satisfying $E u(\xi_i)=0$ and $Eu^2(\xi_i)<\infty$. Since any linear combination of such functions has also expectation zero and a finite second moment, we can use Proposition \ref{normconv} for proving asymptotic normality of our random vector of interest. Let $I(\theta_0)$ denote the Fisher information matrix at $\theta=\theta_0$, that is
$I(\theta_0)$ is the covariance matrix of $\left(\nabla f_{\theta}(\xi_i)\right)_{\theta=\theta_0}/{g(\xi_i)}$.
\begin{proposition} Suppose that $q(t)=h'(t)t$ satisfies the conditions of Proposition \ref{normconvN}. Assume that $E[(\partial f_{\theta}(\xi_i))_{\theta=\theta_0}/g(\xi_i)]^2<\infty$ holds for all the partial derivatives and that the covariance matrix $I(\theta_0)$ is positive definite. Then
\[ \frac{1}{\sqrt{n}} \sum_{i=1}^n \tilde{\psi}_n(\xi_i,\theta_0)=  \frac{1}{\sqrt{n}} \sum_{i=1}^n h'(z_{i,n}(\theta_0))z_{i,n}(\theta_0)
\frac{\left(\nabla f_{\theta}(\xi_i)\right)_{\theta=\theta_0}}{g(\xi_i)}  \]
converges in distribution to a normal distribution with mean zero and with covariance matrix $\sigma_q^2 I(\theta_0)$,
where $\sigma_q^2$ is calculated as in $(\ref{CovConstsigq})$. Observe that $\sigma_q^2$ depends on function $h$ since $q(t)=h'(t)t$.
\end{proposition}
\begin{proof} Let $q(z_{i,n}(\theta_0))=h'(z_{i,n}(\theta_0))z_{i,n}(\theta_0)$. Define $u(\xi_i)$ as
\[ u(\xi_i)= (t_1,\ldots,t_q)  \left(\nabla f_{\theta}(\xi_i)\right)_{\theta=\theta_0}/{g(\xi_i)}. \]
The assertion then follows from Proposition \ref{normconv} by using the Cram\'{e}r-Wold device.
\end{proof}
\begin{proposition} \label{ConvDistrProp}
Consider a random vector $(z_{i,n}(\theta_0),X_n^T)^T$ such that $X_n\stackrel{p}{\to}X$ and the components of $X$ are continuous functions of only $\xi_i$. Then
\[  \left(  \begin{array}{c}
        z_{i,n}(\tho) \\
         {X_n} \\
  \end{array}
\right) \stackrel{D}{\to}  \left(  \begin{array}{c}
        Z \\
        X \\
  \end{array}
\right) \quad \mbox{with} \quad Z\sim Exp(1). \]
\end{proposition}
\begin{proof} We have to show that for any $Z$-continuity set $A_1$ and any $X$-continuity set $A_2$,
\begin{equation} \label{Distrzrest}
P(z_{i,n}(\tho)\in A_1,X_n\in A_2)\to P(Z\in A_1, X\in A_2)=P(Z\in A_1)P(X\in A_2),
\end{equation}
where the last equality holds since $Z$ and $\xi_i$ are independent. Since $X_n\stackrel{p}{\to}X$ and $z_{i,n}(\tho)$ is also independent of $\xi_i$, Theorem 4.3 in \citet{billingsley68} implies that (\ref{Distrzrest}) is the same as
\[P(z_{i,n}(\tho)\in A_1,X\in A_2)= P(z_{i,n}(\tho)\in A_1)P(X\in A_2)\to P(Z\in A_1)P(X\in A_2). \]
\end{proof}
\begin{lemma} \label{UnifLambda} Suppose assumptions

i)  A2, A3, A4

ii) A2, A3

\noindent are fulfilled. Then $\la_n(\thg)\to \la(\thg)$  uniformly for $\thg \in \Theta_0$ as $n\to \infty$ for the functions
i) $h_1$, $h_2$, $h_5$ with $\alpha \in (0,1)$, and ii) $h_5$ with $\alpha \in (1,2]$, respectively.
\end{lemma}
\noindent The proof of Lemma \ref{UnifLambda} is given in the Appendix.
Suppose that $-|\la(\thg)|$ has a unique maximum at $\theta_0$ in $\Theta_0$. This holds for example under the following weak identifiability condition:
\begin{equation} \label{weakTheta0n}
\mu\{ x: \, f_{\thg}(x) \neq f_{\theta_0}(x)  \} >0 \quad \mbox{for} \quad \theta \in \Theta_0,
\end{equation}
where $\mu$ is Lebesgue measure. Then it follows from Lemma \ref{UnifLambda} that $\thon \to \theta_0$. In the following we assume that (\ref{weakTheta0n}) is fulfilled.

\begin{proposition} \label{Prop6} Suppose assumptions

i)  A2, A3, A4, A5

ii) A2, A3, A4, A5, A7

iii) A2, A3, A7

\noindent are fulfilled.
Then
\[ \frac{1}{\sqrt{n}} \sum_{i=1}^n {\psi}_n(\xi_i,\thon)- \frac{1}{\sqrt{n}} \sum_{i=1}^n \tilde{\psi}_n(\xi_i,\theta_0) \stackrel{p}{\to} 0 \]
holds for the functions i) $h_1$, $h_5$ with $\alpha \in (0,1)$, ii) $h_2$, iii) $h_5$ with $\alpha \in (1,2]$, respectively.
\end{proposition}
\begin{proof} Since we are considering convergence in probability, there is no restriction to assume that the studied parameter is one-dimensional (corresponds to looking at the components separately).
Let $A_{i,n}={\psi}_n(\xi_i,\thon)-\tilde{\psi}_n(\xi_i,\theta_0)$. Observe that $E(A_{i,n})=0$ and
\[ E\left( \frac{1}{\sqrt{n}} \sum_{i=1}^n A_{i,n} \right)^2 = E(A_{1,n}^2)+(n-1)E(A_{1,n}A_{2,n}). \]
But $E(A_{1,n}A_{2,n})=\sum_{i=1}^5 P(D_i) E(A_{1,n}A_{2,n}|D_i)$, cf.~the proof of Proposition \ref{normconv}. Given $D_5$, the variables $A_{1,n}$ and $A_{2,n}$ are independent, thus $E(A_{1,n} A_{2,n}| D_5)=0$. For $i=1,\ldots,4$,
$P(D_i)=\mathcal{O}(1/n)$ and $E(A_{1,n} A_{2,n} |D_i)\le E(A_{1,n}^2| D_i)=E(A_{1,n}^2)$. Therefore, the assertion follows if $E(A_{1,n}^2)\to 0$.  Write $A_{1,n}$ as
\[ A_{1,n}= h'(z_{1,n}(\thon))z_{1,n}(\theta_0) \frac{1}{P_{\theta_0}(B_n(\xi_1)} \int_{B_n(\xi_1)} \partial f_{\thon}(y)dy  \]
\[ - h'(z_{1,n}(\theta_0))z_{1,n}(\theta_0)\frac{\partial f_{\theta_0}(\xi_1)}{g(\xi_1)}, \]
and recall that
\[ z_{1,n}(\thon)= z_{1,n}(\theta_0) \frac{1}{P_{\theta_0}(B_n(\xi_1))} \int_{B_n(\xi_1)} f_{\thon} (y) dy. \]
Proposition \ref{ConvDistrProp} together with Lemma \ref{UnifLambda} imply that $A_{1,n}\stackrel{D}{\to} 0$ and $A_{1,n}^2\stackrel{D}{\to} 0$. Thus, $E(A_{1,n})^2 \to 0$ follows because under our assumptions the random variables  ${\psi}_n^2(\xi_1,\thon)$ and $\tilde{\psi}_n^2(\xi_1,\theta_0)$ are uniformly integrable.
\end{proof}
%
%
\section{Asymptotic normality of GMSP estimate via stochastic differentiability}
To prove asymptotic normality of $\hat{\thg}_n$, we need to use a stochastic differentiability condition similar to \citet{pollard85} and \cite{huber67}.
We will prove that
\begin{equation} \label{resConv1}
 \sup_{\Thn} \frac{\left|\juur \sum_i (\psig -\psion) -\sqrt{n} \lambda_n(\thg) \right|}{1+\sqrt{n}|\thg-\tho^{(n)}|} \stackrel{p}{\to} 0, \end{equation}
where $\Thn$ is a compact set shrinking to $\theta_0$ as $n\to \infty$.
%

To prove (\ref{resConv1}), we will consider the numerator of the expression in (\ref{resConv1}) separately on a compact set $K\subset \mathbb{R}^d$  and its complement $K^c$, and show that the contribution from  $K^c$ is arbitrarily small when choosing $K$ large enough. Let
\[ \psigK=\psig I_K(\xi_i), \quad \psigKc=\psig I_{K^c}(\xi_i),  \]
\[ \quad E[{\psigK}] =\ltgK, \quad E[\psigKc]=\lambda_n(\theta,K^c).  \]
Consider the following decomposition of the numerator in (\ref{resConv1}):
\[ \juur \sum_i (\psig -\psion) -\sqrt{n} \lambda_n(\thg)\]
\[= \juur \sum_i (\psigK -\psionK)-\sqrt{n}(\ltgK-\la_n(\tho^{(n)},K)) \]
\[ + \juur \sum_i (\psigKc -\psionKc)- \sqrt{n}(\la_n(\thg,K^c)-\la_n(\tho^{(n)},K^c)). \]
We are going to show the following:

\noindent 1) $\forall \varepsilon>0$, a compact set $K \subset \mathbb{R}^d$ can be chosen so that for large $n$,
\begin{equation} \label{RestCompKc}
P\left( \sup_{\Thn} \frac{\left|\juur \sum_i (\psigKc - \psionKc)-\sqrt{n}(\la_n(\thg,K^c)-\la_n(\tho^{(n)},K^c))\right|}{1+\sqrt{n}|\thg-\tho^{(n)}|}
 >\varepsilon \right) <\varepsilon,
\end{equation}

\noindent 2) for any compact set $K \subset \mathbb{R}^d$,
 \begin{equation} \label{resConv11}
 \sup_{\Thn} \frac{\left|\juur \sum_i (\psigK - \psionK) -\sqrt{n}(\ltgK-\la_n(\tho^{(n)},K))\right|}{1+\sqrt{n}|\thg-\tho^{(n)}|} \stackrel{p}{\to} 0. \end{equation}
Therefore, (\ref{RestCompKc}) and (\ref{resConv11}) together imply (\ref{resConv1}).

Let $V_n(\theta)$ denote the following matrix of partial derivatives:
$$V_n(\theta) = (V_n^{(j,l)}(\theta))=\left(\frac{\partial \lambda_{n,j}(\theta)}{\partial \theta_l} \right), \quad j,l=1,\ldots,q,$$
where $\lambda_{n,j}(\theta)$ is the $j$th element of the vector $\lambda_n(\theta)$. Recall that $\psi_{n,j}(\xi,\thg)$ and $\psi_{j}(\xi,\thg)$ denote the $j$th component of the vectors $\psi_n(\xi,\thg)$ and $\psi(\xi,\thg)$, respectively. Let $V(\thg)=(V^{(j,l)}(\thg))$ with $V^{(j,l)}(\thg)=E\left[ \frac{\partial}{\partial \theta_l} \psi_{j}(\xi,\thg)\right]$, $j,l=1,\ldots,q$.
\begin{lemma} \label{lemma2} Suppose assumptions

i)  A6, A8, A9

ii) A6, A8, A10

iii) A3, A10

\noindent are fulfilled. Then the following assertions hold for i) $h_1$, $h_5$ with $\alpha \in (0,1)$, ii) $h_2$ and iii) $h_5$ with $\alpha \in (1,2]$, respectively. In a neighbourhood of $\theta_0$, $\lambda_n(\theta)$ is continuously differentiable. Furthermore, $V_n(\thg) \to V(\thg)$ uniformly for $\thg \in \Theta_0$ as $n\to \infty$.
\end{lemma}
\begin{proof}
The assumptions of the lemma ensure uniform integrability of the random variables $\sup_{\Theta_0} \left|\frac{\partial}{\partial \theta_l} \psi_{n,j}(\xi,\thg)\right|$.
Thus,
\begin{equation} \label{VnIntegr}
E\left[ \sup_{\Theta_0} \left| \frac{\partial}{\partial \theta_l} \psi_{n,j}(\xi,\thg)\right| \right] <\infty.
\end{equation}
Therefore, we can differentiate under the integral sign and
\[ \frac{\partial}{\partial \theta_l} \lambda_{n,j}(\theta) =E\left[ \frac{\partial}{\partial \theta_l} \psi_{n,j}(\xi,\thg)\right]. \]
Since $\frac{\partial}{\partial \theta_l} \psi_{n,j}(\xi,\thg)$ are continuous functions of $\thg$, it follows from (\ref{VnIntegr}) and the Lebesgue dominated convergence theorem that $V_n(\thg)$ is continuous in $\thg$.
Proposition \ref{ConvDistrProp} implies $\frac{\partial}{\partial \theta_l} \psi_{n,j}(\xi,\thg) \stackrel{D}{\to} \frac{\partial}{\partial \theta_l} \psi_j(\xi,\thg) $. The uniform integrability gives $V_n(\thg) \to V(\thg)$ for every $\theta \in \Theta_0$. The uniform convergence of $V_n(\thg)$ can be proved in the same way as the uniform convergence of $\la_n(\thg)$ in Lemma \ref{UnifLambda}.
\end{proof}
\begin{proposition} \label{HuberKc}
Suppose the assumptions of Lemma \ref{lemma2} hold. Then $\forall \varepsilon>0$ a compact set $K\subset \mathbb{R}^d$ can be chosen so that (\ref{RestCompKc}) holds
for large $n$ for $h_1$, $h_2$ and $h_5$ with $\alpha \in (0,1)\cup(1,2]$.
\end{proposition}
\begin{proof}
Since $\psi_n$ is a vector and the Euclidean norm of a vector is smaller than the sum of the absolute values of its components, it is equivalent to work with single components of the vector and show that the contribution from each component is small. Applying the mean value theorem we obtain:
\[ \frac{\left|\juur \sum_i (\psigKcj - \psionKcj)-\sqrt{n}(\la_{n,j}(\thg,K^c)-\la_{n,j}(\tho^{(n)},K^c))\right|}{1+\sqrt{n}|\thg-\tho^{(n)}|}  \]
\[ \le \frac{1}{n} \sum_{i=1}^n \sum_{l=1}^q \left| \frac{\partial}{\partial \theta_l} \psi_{n,j}(\xi_i,\tilde{\thg},K^c) - V_n^{(j,l)}(\tilde{\thg},K^c)  \right| \]
\[ \le \frac{1}{n} \sum_{i=1}^n \sum_{l=1}^q \sup_{\Thn} \left| \frac{\partial}{\partial \theta_l} \psi_{n,j}(\xi_i,{\thg},K^c) - V_n^{(j,l)}({\thg},K^c)  \right|.   \]
Thus,
\[  P\left( \sup_{\Thn} \frac{\left|\juur \sum_i (\psigKcj - \psionKcj)-\sqrt{n}(\la_{n,j}(\thg,K^c)-\la_{n,j}(\tho^{(n)},K^c))\right|}{1+\sqrt{n}|\thg-\tho^{(n)}|}
 >\varepsilon \right)\]
 \[ \le \frac{1}{\varepsilon} \sum_{l=1}^q E\left[ \sup_{\Thn} \left| \frac{\partial}{\partial \theta_l} \psi_{n,j}(\xi_i,{\thg},K^c) -  V_{n}^{(j,l)}({\thg},K^c)\right|\right] \]
\[ \le \frac{2}{\varepsilon} \sum_{l=1}^q E\left[ \sup_{\Thn} \left| \frac{\partial}{\partial \theta_l} \psi_{n,j}(\xi_i,{\thg},K^c) \right| \right] < \varepsilon \]
if $K$ is large and if $n>n_0$ for some $n_0$.
\end{proof}

To prove (\ref{resConv11}), we will use Lemma 4 in \citet{pollard85}, which is based on bracketing technique, see \cite{vaart2000} and \cite{pollard85}. The bracketing condition enables to divide the parameter set of interest into a finite number of subsets and study the supremum of interest over a finite number of smaller parameter sets. We need also to use the following property of the radii of our nearest neighbour balls: $R_n(i) \stackrel{a.s.}{\to} 0$ for every $i$. Therefore, according to Egoroff's theorem there exists for each $i$ a set $A_i$ with $P(A_i)>1-\frac{\varepsilon}{2}2^{-i}$ such that $R_n(i)\to 0$ uniformly on $A_i$. Therefore, we can define a set $A=\bigcap_{i=1}^{\infty} A_i$, such that $P(A^c)<\varepsilon/2$.
\newline \newline
\textbf{Bracketing}.
 Lemma 4 in \citet{pollard85} will be applied to functions in $$\mathcal{F}_n=\{ [\psigK-\psionK]I_{A}, \, \, |\thg-\tho|\le \alpha_n, \, i=1,\ldots,n\}.$$
Since $[\psigK-\psionK]I_{A}$, $i=1,\ldots,n$, are identically distributed, we can suppress $i$ in $\xi_i$, $B_n(\xi_i)$, $z_{i,n}(\theta)$ and $R_n(i)$ right now. That the bracketing condition is fulfilled follows since the functions $[{\psi}_n(\xi,\theta,K) - {\psi}_n(\xi,\theta_0^{(n)},K)]I_{A}$ satisfy a
Lipschitz condition
\[ |{\psi}_n(\xi,\theta_s,K) - {\psi}_n(\xi,\theta_t,K)| I_{A} \le H_n(\xi,K)|\theta_s - \theta_t|, \]
where
\[ H_n(\xi,K) =q \left[ \sup_{\Thn} |h''(z_n(\theta)) z_n^2(\theta)| (f_{\rm{max}}^{(1)}(\xi))^2  +
 \sup_{\Thn} |h'(z_n(\theta))z_n(\theta)| f_{\rm{max}}^{(2)}(\xi)\right]I_{A}I_K(\xi), \]
 with
\[ f_{\rm{max}}^{(1)}(\xi)=\max_l \sup_{\Thn} \left( \frac{1}{P_{\theta}(B_n)}\int_{B_n} \left|\frac{\partial f_{\theta}(y)}{\partial \theta_l}  \right| dy \right), \quad \]
\[f_{\rm{max}}^{(2)}(\xi)=\max_{j,l} \sup_{\Thn} \left( \frac{1}{P_{\theta}(B_n)}\int_{B_n} \left|\frac{\partial^2  f_{\theta}(y)}{\partial \theta_j \partial \theta_l} \right| dy \right),  \]
and where for some constant $b_1(h,K)$, $E[H_n(\xi,K)] < b_1(h,K) <\infty$ when $n$ is large enough.
\begin{proposition} \label{HuberK}
Consider a compact set $K \subset \mathbb{R}^d$. Suppose assumption A2 is fulfilled. Then the family $\mathcal{F}_n$ satisfies the bracketing condition
and $E H_n^2(\xi,K)< b_2(h,K)<\infty$ holds for some constant $b_2(h,K)$ and for large $n$. Therefore, the convergence in (\ref{resConv11}) holds for $h_1$, $h_2$ and $h_5$ with $\alpha \in (0,1)\cup(1,2]$.
\end{proposition}
\begin{proof} For the bracketing condition to be fulfilled we need to show that  $E[H_n(\xi,K)] < b_1(h,K) <\infty$. Since this follows from $E[H_n(\xi,K)]^2 < b_2(h,K) <\infty$, we are going to prove that
\begin{equation} \label{brackE1} E\left[\sup_{\Thn} |h''(z_n(\theta))z_n^2(\theta)| [f_{\rm{max}}^{(1)}(\xi)]^2 \cdot I_{A} I_K(\xi)\right]^2 < \infty,  \end{equation}
\begin{equation} \label{brackE2} E\left[ \sup_{\Thn} |h'(z_n(\theta)) z_n(\theta)| f_{\rm{max}}^{(2)}(\xi) \cdot I_{A} I_K(\xi)\right]^2 < \infty. \end{equation}
Define the closed $\delta$-neighbourhood $K_{\delta}=\{x\in \mathbb{R}^d: \, d(x,K)\le \delta\}$, where $d(x,K)=\inf\{d(x,y): \, y\in K\}$.
When $\xi \in K$ and $\omega\in A$, we have for $n$ large enough that $B_n \subset K_{\delta}$.
Therefore, for large $n$,
\[  \sup_{\Thn} \left( \frac{1}{P_{\thg}(B_n)} \int_{B_n} \frac{|\partial f_{\thg}(y)|}{f_{\thg}(y)} dP_{\thg}(y)   \right)^4 I_K(\xi)I_{A} \]
\[ \le \sup_{\Thn} \left( \frac{|\partial f_{\thg}(\xi)|}{f_{\thg}(\xi)} +\eta \right)^4 I_K(\xi)I_{A}\le N_K,  \]
where the last inequality holds because ${|\partial f_{\thg}(y)|}/{f_{\thg}(y)}$ is uniformly continuous on $K\times \Thn$. In a similar way we obtain
\[ \left(\sup_{\Thn} \frac{1}{P_{\thg}(B_n)} \int_{B_n} |\partial^2f_{\thg}(y)| dy \cdot I_K(\xi)I_{A}\right)^2  \le \sup_{\Thn} \left( \frac{|\partial^2 f_{\thg}(\xi)|}{f_{\thg}(\xi)}  +\eta \right)^2 I_K(\xi)I_{A} <M_K. \]
Since $z_n(\theta)=z_n(\theta_0)\cdot \frac{1}{P_{\theta_0}(B_n)}\int_{B_n} f_{\theta}(y) dy$ and $z_n(\theta_0)$ has moments of all orders, (\ref{brackE1}) and (\ref{brackE2}) follow for our functions $h_1$, $h_2$ and $h_5$.
The finite expectation $EH^2_n(\xi,K)<b_2(h,K)$ implies
 \[ E\left[ \sup_{\Thn} |\psi_n(\xi_i,\theta,K)-\psi_n(\xi_i,\theta_0^{(n)},K)|^2I_{A} \right] \le E[H_n(\xi_i,K)]^2 \alpha_n^2 \to 0 \quad \mbox{as} \quad n\to \infty,   \]
and that the bracketing functions have finite variance. Moreover, in the same way as in the proof of Proposition \ref{Prop6} it can be shown that
\[ \mbox{Var}\left(\frac{1}{\sqrt{n}} \sum_i H_n(\xi_i,K)\right) =\mathcal{O}(\mbox{Var}(H_n(\xi_1,K))), \]
see also p.~306 and the proof of Lemma 4 in \citet{pollard85}.
As $P(A^c)<\varepsilon/2$, (\ref{resConv11}) follows due to Lemma 4 in \citet{pollard85}.
\end{proof}
As the last step we will use Lemma 3 in \cite{huber67} and prove the asymptotic normality of $\hat{\thg}_n$.
\begin{theorem} Let $\hat{\theta}_n \stackrel{p}{\to} \tho$ hold. Suppose $I(\theta_0)$ is positive definite and the assumptions of Proposition 6, 7 and 8 are satisfied.
Then
\[ \sqrt{n}(\hat{\thg}_n-\thon) \stackrel{D}{\to} \mathcal{N}\left(0,\frac{\sigma_q^2}{b_h^2}I(\theta_0)^{-1} \right),    \]
where $b_h=E[h''(Z)Z^2]$.
\end{theorem}
\begin{proof}
Propositions \ref{HuberKc} and \ref{HuberK} imply
\[
 \sup_{\Thn} \frac{\left|\juur \sum_i (\psig -\psion) -\sqrt{n} \lambda_n(\thg) \right|}{1+\sqrt{n}|\thg-\thon|} \stackrel{p}{\to} 0. \]
Applying the mean value theorem to $\lambda_{n,j}(\theta)$ we obtain that there exists $\tilde{\theta}_{n,j}$ such that
$$\lambda_{n,j}(\thg)=V_n^{(j)}(\tilde{\theta}_{n,j})(\thg-\thon), \quad j=1,\ldots,q,$$
where $V_n^{(j)}(\cdot)$ denotes the $j$th row of the matrix $V_n(\cdot)$. Define a matrix $V_n^{\ast}(\thg)$, where the rows are given by $V_n^{(j)}(\tilde{\theta}_{n,j})$, $j=1,\ldots,q$. We use the argument $\thg$ in $V_n^{\ast}$ to indicate that it comes from an application of the mean value theorem to the components of $\lambda_n(\thg)$. As $V_n^{\ast}(\thg)\to V(\theta_0)=b_h I(\theta_0)$ uniformly for $\thg \in \Thn$, it follows that $V_n^{\ast}(\thg)$ is invertible for large $n$ and $\theta \in \Thn$. Therefore, for some $C>0$,
\[ |\thg-\thon| \le |V_n^{\ast}(\thg)^{-1}||\lambda_n(\thg)| \le (|V(\theta_0)^{-1}|+\varepsilon)|\lambda_n(\thg)| =C |\lambda_n(\thg)|.  \]
It follows that
\[  \sup_{\Thn} \frac{\left|\juur \sum_i (\psig -\psion) -\sqrt{n} \lambda_n(\thg) \right|}{1+\sqrt{n}|\lambda_n(\thg)|} \stackrel{p}{\to} 0, \]
which corresponds to Lemma 3 in \cite{huber67}. Applying Theorem 3 in \cite{huber67} and using the consistency of $\hat{\thg}_n$ gives
\[ \frac{1}{\sqrt{n}} \sum_i \psion +\sqrt{n}\lambda_n(\hat{\thg}_n) \stackrel{p}{\to} 0, \]
where $\frac{1}{\sqrt{n}} \sum_i \psion$ is asymptotically normally distributed with mean zero and covariance matrix $\sigma_q^2 I(\theta_0)$.
It follows that $\sqrt{n}\lambda_n(\hat{\thg}_n) \stackrel{D}{\to} \mathcal{N}(0,\sigma_q^2 I(\theta_0))$.
Applying the mean value theorem again gives that for some $\tilde{\theta}_{n,1},\ldots,\tilde{\theta}_{n,q}$ depending on $\hat{\thg}_n$, we can define a matrix $V_n^{\ast}(\hat{\thg}_n)$ so that  $\sqrt{n}\lambda_n(\hat{\thg}_n)=V_n^{\ast}(\hat{\thg}_n)\sqrt{n}(\hat{\thg}_n-\thon)$. As
$\sqrt{n}(\hat{\thg}_n-\thon)=V_n(\theta_0)^{-1}V_n(\theta_0)\sqrt{n}(\hat{\thg}_n-\thon)$, we obtain
\[|\sqrt{n}(\hat{\thg}_n-\thon)|\le |V_n(\theta_0)^{-1}| (|V_n^{\ast}(\hat{\thg}_n)\sqrt{n}(\hat{\thg}_n-\thon)| +|V_n(\theta_0)-V_n^{\ast}(\hat{\thg}_n)||\sqrt{n}(\hat{\thg}_n-\thon)|) \]
\[ \le |V_n(\theta_0)^{-1}| \left( \mathcal{O}_P(1)+o_P(1)\sqrt{n}|\hat{\thg}_n-\thon|  \right). \]
Thus, $ |\sqrt{n}(\hat{\thg}_n-\thon)| \le \mathcal{O}_P(1)$. It follows that
\[ \sqrt{n}\lambda_n(\hat{\thg}_n) =V_n(\theta_0)\sqrt{n}(\hat{\thg}_n-\thon) +o_P(1)=V(\theta_0)\sqrt{n}(\hat{\thg}_n-\thon) +o_P(1), \]
and thus,
$$\sqrt{n}(\hat{\thg}_n-\thon) \stackrel{D}{\to}\mathcal{N}(0, V(\thg_0)^{-1}\sigma_q^2 I(\theta_0) [V(\thg_0)^{-1}]^{T}).$$
Using that $V(\tho)=b_h I(\theta_0)$, the assertion follows.
\end{proof}
\section{Discussion}
For univariate spacings asymptotic normality of GMSP estimators has been shown in \citet{GhoshJam}. Recently, \citet{luong} also considered consistency and asymptotic normality of univariate GMSP estimates. Since the author has overlooked the local dependence between nearest neighbours, the proof of asymptotic normality in \citet{luong} is not correct and thus also the derived asymptotic variance is incorrect. \citet{GhoshJam} showed that the smallest variance in the asymptotic distribution was obtained for $h(x)=\ln(x)$ and that this smallest variance coincides with the Cram\'{e}r-Rao lower bound. We have calculated the constants ${\sigma_q^2}/{b_h^2}$ in the asymptotic covariance matrix for the $h$-functions studied in this article, see Table \ref{VarConst}. The smallest variance is obtained for
$h_1(x)=\ln x-x+1$. For $h_5 (x) = \mbox{sgn}(1-\alpha)(x^{\alpha}-\alpha x +\alpha -1)$ the variance increases with increasing values of $\alpha$ and
when $\alpha \searrow 0$, the variance tends to the variance of $h_1$.
\newline \newline
Table \ref{VarConst} here.
\newline \newline
In this article, we have proved asymptotic normality of the generalized maximum spacing estimate $\hat{\thg}_n$ around $\thon$, where $\thon$ maximizes the expectation function $E[h(z_{i,n}(\thg))]$.
For the asymptotic normality to hold around $\theta_0$, it has to be shown that $\sqrt{n}(\thon-\theta_0)\to 0 $. According to the mean value theorem, for some constant $C>0$,
\[ \sqrt{n}|\thon-\theta_0| \le C\sqrt{n} |\lambda_n(\theta_0)|.    \]
Thus, if $\sqrt{n} |\lambda_n(\theta_0)|\to 0$, then $\sqrt{n}(\thon-\theta_0)$ follows. The behaviour of $\lambda_n(\theta_0)$ depends on what parameters are considered. In the case of multivariate normal distribution $\mathcal{N}_d(\mu_0,\Sigma_0)$ it follows because of symmetry that $E[h(z_{i,n}(\thg))]$ is maximized by $\mu_0$ regardless of $\Sigma_0$. Thus, $\mu_0^{(n)}=\mu_0$.
For bivariate normal distribution with $\theta=(\mu_1,\mu_2,\sigma_1,\sigma_2,\rho)$ we have studied the behaviour of $\lambda_n(\theta_0)$ in simulation studies for the following parameter vector:
$\theta_0=(1,2,1,1,0.5)$. We simulated a sample of $n$ observations from this distribution and calculated for a randomly chosen observation in the sample the quantity
\[ h'(z_{i,n}(\theta_0))z_{i,n}(\theta_0) \frac{1}{P_{\theta_0}(B_n(\xi_i))} \int_{B_n(\xi_i)} (\nabla f_{\thg} (y))_{\theta=\theta_0} dy.  \]
This procedure was repeated for $m=10000$ samples and $\lambda_n(\theta_0)$ was estimated with the average of the 10000 values. In Table $\ref{compsig1}$, the mean values over 20 repetitions are presented for the component of $\sqrt{n}\lambda_n(\theta_0)$ that corresponds to $\sigma_1$. We calculated the values for $h_1$, $h_2$ and $h_3$ with $p=2$.
It can be seen in Table $\ref{compsig1}$ that for all the considered $h$-functions the estimated values of $\sqrt{n}\lambda_n(\theta_0)$ for the component corresponding to $\sigma_1$ decrease when $n$ increases and approach slowly zero. The same behaviour can be observed for $\sigma_2$ and the correlation parameter $\rho$. Thus, the simulation results indicate that $\sqrt{n}\lambda_n(\theta_0)\to 0$ as $n\to \infty$ for the components corresponding to $\sigma_1$, $\sigma_2$ and $\rho$, although the convergence is very slow.
\newline \newline
Table \ref{compsig1} here.
\newline \newline
For most distributions, it is not difficult to check assumptions A3, A7, A9 and A10. Remark 1 is useful for checking assumptions A4, A6 and A8.
It is not difficult to verify that our assumptions are satisfied for the class of multivariate normal distributions and for finite mixtures of normal distributions.
\subsection*{Corresponding author} e-mail: kristi.kuljus@ut.ee; address: Institute of Mathematics and Statistics, J.~Liivi 2, 50 409, Tartu, Estonia
\subsection*{Acknowledgments}
This work is supported by Estonian institutional research funding IUT34-5.
\newpage
\begin{table} [h!]
\centering
\caption{Constants ${\sigma_q^2}/{b_h^2}$ in the asymptotic covariance matrix for different $h$-functions. \label{VarConst}}
\begin{tabular}{|l|r|}
  \hline
    $h(x)$ &  ${\sigma_q^2}/{b_h^2}$ \\
  \hline
    $h_1(x)=\ln x-x+1$ &  1.8434    \\
    $h_2(x)=(1-x)\ln x$ & 2.2130     \\
    $h_5(x)=\mbox{sgn}(1-\alpha)(x^{\alpha}-\alpha x +\alpha -1)$  &  $ \alpha=0.1 \quad$ 1.9265 \\
                      & $ \alpha=0.5 \quad$ 2.3421  \\
                      & $ \alpha=0.9 \quad$ 2.7493  \\
                      & $ \alpha=2 \quad \,\,\,\,\,  $ 3.6546    \\
  \hline
\end{tabular}
\end{table}
\begin{table} [h!]
\centering
\caption{Estimated values of $\sqrt{n} \lambda_n(\theta_0)$ for the component corresponding to $\sigma_1$ in the case of bivariate normal distribution with the parameters $(\mu_1,\mu_2,\sigma_1,\sigma_2,\rho)=(1,2,1,1,0.5)$. \label{compsig1}}
\begin{tabular}{|c|c|c|c|}
  \hline
   $\sqrt{n}$  & $h_1(x)=\ln x-x+1$ & $h_2(x)=(1-x)\ln x$ & $h_3(x)=-(1-\sqrt{x})^2$ \\
  \hline
   10 & 0.3910 & 1.1613 & 0.2769 \\
   30 & 0.3641 & 1.0580 & 0.2514 \\
   40 &  0.3399 &   1.0025  &  0.2368 \\
   60 &  0.2796  &  0.8124  &   0.1924 \\
   100 & 0.2213  &   0.6450  &  0.1527 \\
   200 & 0.1576  &  0.4673  &   0.1101 \\
   500 &  0.1233  &  0.3504  &  0.0836  \\
   700 &  0.0382  &  0.1193  &  0.0282  \\
  \hline
\end{tabular}
\end{table}
\newpage
\section{Appendix}
\noindent \textbf{\emph{Proof of Lemma \ref{UnifLambda}}}.
\begin{proof} Recall that $\la_n(\thg)= E[ \psig ]$ and $\la(\thg)=E[{\psi}(\xi_i,\thg)]$. Thus, we  can suppress $i$ in the notation and write ${\psi}_n(\xi,\thg)$, $\tilde{\psi}_n(\xi,\thg)$,
${\psi}(\xi,\thg)$ and $z_n(\tho)$ for the random quantities of interest. Since uniform convergence can be proved componentwise, we will write $\partial f_{\thg}(\xi)$ to emphasize that the same approach holds for any component
$\partial/\partial \thg_j$, $j=1,\ldots,q$. We suppress $j$ also in vector notations $\psi_{n,j}$, $\psi_j$ etc. The uniform convergence of $\la_n(\thg)$ to  $\la(\thg)$ in $\Theta_0$ holds if
\[ \sup_{\thg \in \Theta_0} |\la_n(\thg) -\la(\thg)| \to 0 \quad \mbox{as} \quad n\to \infty.  \]
Thus, we will study
\[ \sup_{\Theta_0} |E[{\psi}_n(\xi,\thg)-{\psi}(\xi,\thg)]| \le   \sup_{\Theta_0} |E[{\psi}_n(\xi,\thg)-\tilde{\psi}_n(\xi,\thg)] | \]
\[ +  \sup_{\Theta_0} |E[\tilde{\psi}_n(\xi,\thg) -{\psi}(\xi,\thg)]| =\mbox{Term}_{\textrm{I}} +\mbox{Term}_\textrm{{II}}.  \]
We will show that both terms converge to zero under the assumptions of the lemma.
\newline \newline
\noindent \textbf{Term}$_\textbf{\textrm{{II}}}$. Observe that
\[ \tilde{\psi}_n(\xi,\thg) -{\psi}(\xi,\thg) =
 \left[  v\left( z_n(\tho) \frac{f_{\thg}(\xi)}{g(\xi)} \right) - v\left( Z \frac{f_{\thg}(\xi)}{g(\xi)} \right)  \right] \frac{\partial f_{\thg}(\xi)}{f_{\thg}(\xi)}. \]
We will exemplify the proof using $h_3$, the proof is similar for other $h$-functions. For $h_3$ we have $h_3'(z)z=-z+\sqrt{z}$, therefore
\[ \tilde{\psi}_n(\xi,\thg) -{\psi}(\xi,\thg)= [Z-z_n(\tho)] \frac{f_{\thg}(\xi)}{g(\xi)} \frac{\partial f_{\thg}(\xi)}{f_{\thg}(\xi)}
+ \left[\sqrt{z_n(\tho)}-\sqrt{Z}\right] \sqrt{\frac{f_{\thg}(\xi)}{g(\xi)}} \frac{\partial f_{\thg}(\xi)}{f_{\thg}(\xi)}.   \]
Because $\sqrt{f_{\thg}(\xi)}/\sqrt{g(\xi)}\le 1+f_{\thg}(\xi)/g(\xi)$, we obtain
\[ \sup_{\Theta_0} |E[\tilde{\psi}_n(\xi,\thg) -{\psi}(\xi,\thg)]| \le \left| E(z_n(\tho)-Z)\right| \sup_{\Theta_0} E\left[ \frac{|\partial f_{\thg}(\xi)|}{g(\xi)} \right] \]
\[ + \left| E\left(\sqrt{z_n(\tho)}-\sqrt{Z}\right) \right| \sup_{\Theta_0} E\left[ \frac{|\partial f_{\thg}(\xi)|}{f_{\thg}(\xi)} + \frac{|\partial f_{\thg}(\xi)|}{g(\xi)} \right]. \]
Since $z_n(\tho)\stackrel{D}{\to} Z$ and $\sqrt{z_n(\tho)}\stackrel{D}{\to} \sqrt{Z}$, and since the corresponding moments of $z_n(\tho)$ and $Z$ are finite, Term$_\textrm{{II}}\to 0$ follows because of convergence of the respective expected values and because of Assumptions A3 and A4.
\newline \newline
\noindent \textbf{Term}$_\textbf{\textrm{I}}$. According to Proposition \ref{ConvDistrProp}, ${\psi}_n(\xi,\thg)-\tilde{\psi}_n(\xi,\thg)\stackrel{D}{\to} 0$ for each $\thg$.
This implies that for any $\theta_1,\ldots,\theta_m\in \Theta_0$, where $m$ is any finite number, the respective finite-dimensional distribution converges to a zero-vector of length $m$ in distribution. Observe that
${\psi}_n(\xi,\thg)-\tilde{\psi}_n(\xi,\thg)$ are continuous functions of $\thg$ in a neighbourhood of $\tho$. We will prove tightness of $\{ \tilde{\psi}_n(\xi,\thg) \}$ and
$\{{\psi}_n(\xi,\thg)\}$ and use that this together with the convergence of finite-dimensional distributions implies
\[ \sup_{\Theta_0}|{\psi}_n(\xi,\thg)-\tilde{\psi}_n(\xi,\thg)| \stackrel{D}{\to} 0.\]
Since $\sup_{\Theta_0}|{\psi}_n(\xi,\thg)-\tilde{\psi}_n(\xi,\thg)|$ is uniformly integrable due to our assumptions, we then obtain
\[  \sup_{\Theta_0} \left|E[{\psi}_n(\xi,\thg)-\tilde{\psi}_n(\xi,\thg)]\right| \le E\left[ \sup_{\Theta_0} |{\psi}_n(\xi,\thg)-\tilde{\psi}_n(\xi,\thg)| \right] \to 0. \]
Since both ${\psi}_n(\xi,\tho)$ and $\tilde{\psi}_n(\xi,\tho)$ converge to ${\psi}(\xi,\tho)$ in distribution, ${\psi}_n(\xi,\tho)$ and $\tilde{\psi}_n(\xi,\tho)$ are tight according to Prohorov's theorem. To show tightness of $\tilde{\psi}_n(\xi,\thg)$, take now arbitrary $\varepsilon>0$ and $\eta>0$. Choose a compact set $K\subset \mathbb{R}^d$ and a constant $M>0$ such that
$P(\xi \in K)>1-\eta/4$, $\sup_n P(z_n(\tho)>M)<\eta/4$. Consider arbitrary $\thg_1$, $\thg_2$ in $\Theta_0$. Applying the equality
\[ a(\thg_1)b(\thg_1)- a(\thg_2)b(\thg_2) = \left( a(\thg_1)-a(\thg_2) \right) b(\thg_1) +\left( b(\thg_1)-b(\thg_2) \right)a(\thg_2),  \]
we obtain
\[ \left| \tilde{\psi}_n(\xi,\thg_1)-\tilde{\psi}_n(\xi,\thg_2)\right| I_K(\xi)I(z_n(\tho)\le M)  \]
\[   \le \left\{ \left| v\left( z_{n}(\theta_0) \frac{f_{\thg_1}(\xi)}{g(\xi)} \right)- v\left(z_{n}(\theta_0) \frac{f_{\thg_2}(\xi)}{g(\xi)} \right)\right|
\sup_{\Theta_0}\frac{|\partial f_{\thg}(\xi)|}{f_{\thg}(\xi)}  \right.   \]
\[ + \left. \left| \frac{\partial f_{\thg_1}(\xi)}{f_{\thg_1}(\xi)}- \frac{\partial f_{\thg_2}(\xi)}{f_{\thg_2}(\xi)} \right|
\sup_{\Theta_0} \left| v\left(z_{n}(\theta_0) \frac{f_{\thg}(\xi)}{g(\xi)} \right)\right| \right\} I_K(\xi) I(z_n(\tho)\le M). \]
Since our functions of interest are uniformly continuous on $K\times \Theta_0$, we obtain that there exists $\delta_1>0$ such that
\[ \left| \tilde{\psi}_n(\xi,\thg_1)-\tilde{\psi}_n(\xi,\thg_2) \right| I_K(\xi)I(z_n(\tho)\le M) < \varepsilon \]
whenever $|\thg_1-\thg_2|<\delta_1$.

Tightness of $\{{\psi}_n(\xi,\thg)\}$ follows analogously, but now we also need to bring in the set $A$, where $R_n(\omega)\to 0$ uniformly. Therefore, if $n$ is large enough and
$\omega \in A\cap \{ \xi\in K \} \cap \{ z_n(\tho) \le M \}$, there exists $\delta_2$ such that whenever $|\thg_1-\thg_2|<\delta_2$,
\[ \left| \frac{1}{P_{\tho}(B_n)} \int_{B_n} \frac{ f_{\thg_1}(y)}{g(y)} dP_{\tho}(y) - \frac{1}{P_{\tho}(B_n)} \int_{B_n} \frac{f_{\thg_2}(y)}{g(y)} dP_{\tho}(y)  \right|,  \]
\[ \left| \frac{1}{P_{\theta_1}(B_n)} \int_{B_n} {\partial f_{\thg_1}(y)}  dy - \frac{1}{P_{\theta_2}(B_n)} \int_{B_n} {\partial f_{\theta_2}(y)} dy  \right|  \]
become sufficiently small.
\end{proof}
%
%
\bibliographystyle{apa}
\bibliography{GMSPasnorm}
\end{document}